\documentclass[11pt,a4paper,dvips,twoside]{article}

\usepackage{amssymb}
\usepackage{amsmath}
\usepackage{amsthm}
\usepackage{enumerate}
\usepackage{graphicx}
\usepackage{amsfonts}

\newtheorem{theorem}{Theorem}[section]

\newtheorem{remark}{Remark}[section]

\begin{document}

	\begin{center}
		{\Large {\bf Classical Kantorovich operators revisited}}
		\vspace{0.5cm} 
		
		{\large Ana-Maria Acu${}^{1}$, Heiner Gonska${}^{2}$}
		
		\vspace{0.3cm}

	\end{center}
	\vspace{0.3cm}
	
	\begin{abstract}
		The main object of this paper is to improve some of the known estimates for classical Kantorovich operators. A quantitative Voronovskaya-type result in terms of second moduli of continuity which improves some previous results is obtained. In order to explain non-multiplicativity of the Kantorovich operators a Chebyshev-Gr\"uss inequality is given. Two Gr\"uss-Voronovskaya theorems for Kantorovich operators are considered as well.\\
		
	\end{abstract}
		\vspace*{0.3cm}

\section{Introduction}
In 1930 L.V. Kantorovich \cite{K} introduced a significant modification of the classical Bernstein operators given by
$$
K_{n}(f;x)=(n+1)\displaystyle\sum_{k=0}^n{p}_{n,k}(x)\int_{\frac{k}{n+1}}^{\frac{k+1}{n+1}}f(t)dt.
$$
Here $n\geq 1$, $f\in L_1[0,1]$, $x\in[0,1]$ and
\begin{align*}
&p_{n,k}(x)={n\choose k} x^k(1-x)^{n-k},\,\,0\leq k\leq n, \\
&p_{n,k}\equiv  0,\textrm{ if } k<0\textrm{ or }k>n.
\end{align*}
These mappings are relevant since they provide a constructive tool to approximate any function in $L_p[0,1]$, $1\leq p<\infty$, in the $L_p$-norm. For $p=\infty$, $C[0,1]$ has to be used instead of $L_{\infty}[0,1]$.

These classical Kantorovich operators have been attracting a lot of attention since then, but results on them are somehow scattered in the literature. They share this with other relevant variations of the Bernstein-type: Durrmeyer, genuine Bernstein-Durrmeyer and, last but not least, variation-diminishing Schoenberg splines.

In the present note we first collect and improve some of the known estimates by giving quite a precise inequality for $f\in C^{r}[0,1]$, $r\in \mathbb{N}\cup\{ 0\}$, a new Voronovskaya result in terms of $\omega_2$ and a Chebyshev-Gr\"uss inequality giving an explanation of their non-multiplicativity. The last part of this article deals with two Gr\"uss-Voronovskaya theorems for Kantorovich operators.

Most estimates in this article will be given in terms of moduli of smoothness of higher order. In the background, but not explicitly mentioned, is always the K-functional technique. In this sense we were very much influenced by the work of Zygmund (see, e.g., \cite{Zyg}), a hardly accessible conference contribution of Peetre \cite{Pee} and also by the book of Dzyadyk \cite{Dzy}.

\section{Some previous results}
In this section we collect some results given earlier. Quite a strong general result was given by the second author and Xin-long Zhou \cite{7} in 1995.

Let $\varphi(x)=\sqrt{x(1-x)}$ and  $P(D)$ be the differential operator given by
$$ P(D)f:=(\varphi^2 f^{\prime})^{\prime},\,\, f\in C^{2}[0,1]. $$
For $f\in L_p[0,1]$, $1\leq p\leq \infty$, the functional $K(f,t)_p$ is defined as below
$$ K(f,t)_p:=inf\left\{\| f-g\|_p+t^2\| P(D)g\|_p: g\in C^2[0,1]\right\}. $$
Using the above functional in \cite{7}  the following theorem was proved.
\begin{theorem}\label{t1} There exists an absolute positive constant $C$ such that for all $f\in L_p[0,1]$, $1\leq p\leq\infty$, there holds
	$$ C^{-1}K(f,n^{-1/2})_p\leq\|f-K_nf\|_p\leq CK(f,n^{-1/2})_p.   $$
	\end{theorem} 
Also, in order to characterize the $K$-functional used in Theorem \ref{t1}, the next result was given in \cite{7}:

\begin{theorem}
	We have
	$$ K(f,t)_p\sim \omega_{\varphi}^2(f,t)_p+t^2E_0(f)_p,\,\, 1<p<\infty,   $$
	and
	$$ K(f,t)_{\infty}\sim\omega_{\varphi}^2(f,t)_{\infty}+\omega(f,t^2)_{\infty}. $$
	Here  $\omega(f,t)_p$ is the classical modulus, $\omega_{\varphi}^2(f,t)_{\infty}$ denotes the second order modulus of smoothness with weight function $\varphi$ and $E_0(f)_p$ is the best approximation constant of $f$ defined by
	$$ E_0(f)_p=\inf_{c}\|f-c\|_p. $$	
\end{theorem}

Moreover, all quantities subscripted by $\infty$ are taken with respect to the uniform norm in $C[0,1]$. The following theorem of P\u alt\u anea \cite{8} is the key to give a more explicit result in terms of classical moduli for continous functions. See \cite{5} for details.
\begin{theorem}\cite{8} If $L:C[0,1]\to C[0,1]$ is a positive linear operator, then for $f\in C[0,1], x\in[0,1]$ and each $0<h\leq \dfrac{1}{2}$ the following holds:
	\begin{align*}
	|L(f;x)-f(x)|&\leq |L(e_0;x)-1|\cdot |f(x)|+\dfrac{1}{h}|L(e_1-x;x)|\omega(f;h)\\
	&+\left[(Le_0)(x)+\dfrac{1}{2h^2}L((e_1-x)^2;x)\right]\omega_2(f;h).
	\end{align*}
	The condition $h\leq 1/2$ can be eliminated for operators $L$ reproducing linear functions.
	\end{theorem}
\begin{theorem}\label{t1.4}
	For all $f\in C[0,1]$ and all $n\geq 4$, 
	$$ \|K_nf-f \|_{\infty}\leq \dfrac{1}{2\sqrt{n}}\omega_1\left(f;\dfrac{1}{\sqrt{n}}\right)+\dfrac{9}{8}\omega_2\left(f;\dfrac{1}{\sqrt{n}}\right).  $$
\end{theorem}
This result can be extended to simultaneous approximation, see again \cite{5}
\begin{theorem}
	Let $r\in \mathbb{N}_0,n\geq 4, f\in C^r[0,1]$. Then
	\begin{align*}
	\| D^rK_nf-D^rf \|_{\infty}\leq\dfrac{(r+1)r}{2n}\| D^rf\|_{\infty}+\dfrac{r+1}{2\sqrt{n}}\omega_1\left(D^rf;\dfrac{1}{\sqrt{n}}\right)+\dfrac{9}{8}\omega_2\left(D^rf;\dfrac{1}{\sqrt{n}}\right).
	\end{align*}
\end{theorem}

\section{A quantitative Voronovskaya result }

This part has its predecessor in a hardly known booklet of Videnskij in which a quantitative version of the well-known Voronovskaya theorem for the classical Bernstein operators can be found (see \cite{Vid}). This estimate was generalized and improved in \cite{6}. An application
for Kantorovich operators was given in \cite{5}. Here we improve it as follows:

\begin{theorem}\label{t1.6} For $n\geq 1$ and $f\in C^2[0,1]$, one has
		\begin{align}\label{ec3}
	&	\left\|n\left(K_nf-f\right)-\dfrac{1}{2}\left(Xf^{\prime}\right)^{\prime} \right\|_{\infty}\leq \dfrac{2}{3(n+1)}\left(\dfrac{3}{4}\|f^{\prime}\|_{\infty}+\|f^{\prime\prime}\|_{\infty}\right)\nonumber\\
	&+\dfrac{9}{32}\left\{\dfrac{2}{\sqrt{n+1}}\omega_1\left(f^{\prime\prime};\dfrac{1}{\sqrt{n+1}}\right)+\omega_2\left(f^{\prime\prime};\dfrac{1}{\sqrt{n+1}}\right)\right\},
	\end{align}
	where $X=x(1-x)$ and $X^{\prime}=1-2x$, $x\in[0,1]$.
	\end{theorem}
	\begin{proof}
	From \cite[Theorem 3]{6} we get
		\begin{align*}
		&\left| K_n(f;x)-f(x)-K_n(t-x;x)f^{\prime}(x)-\dfrac{1}{2}K_n\left((e_1-x)^2;x\right)f^{\prime\prime}(x)\right|\\
		&\leq K_n((e_1-x)^2;x)\left\{\dfrac{|K_n((e_1-x)^3;x)|}{K_n((e_1-x)^2;x)}\dfrac{5}{6h}\omega_1(f^{\prime\prime};h)+\left(\dfrac{3}{4}+\dfrac{K_n((e_1-x)^4;x)}{K_n((e_1-x)^2;x)}\cdot\dfrac{1}{16h^2}\right)\omega_2(f^{\prime\prime};h)\right\}.
		\end{align*}
		Using the central moments up to order 4 for Kantorovich operators, namely
		\begin{align*}
		&K_n\left(t-x;x\right)=\dfrac{1-2x}{2(n+1)},\\
		& K_n\left((t-x)^2;x\right)=\dfrac{1}{(n+1)^2}\left\{x(1-x)(n-1)+\frac{1}{3}\right\},\\
		&K_n\left((t-x)^3;x\right)=\dfrac{1-2x}{4(n+1)^3}\left\{10x(1-x)n+2x^2-2x+1\right\},\\
		&K_n\left((t-x)^4;x\right)=\dfrac{1}{(n+1)^4}\left\{3x^2(1-x)^2n^2+5x(1-x)(1-2x)^2n+x^4-2x^3+2x^2-x+\dfrac{1}{5}\right\},
		\end{align*}
		we obtain
		\begin{align*}
		\dfrac{|K_n\left((t-x)^3;x\right)|}{K_n\left((t-x)^2;x\right)}\leq \dfrac{5}{2(n+1)};\,\,\,\, \dfrac{|K_n\left((t-x)^4;x\right)|}{K_n\left((t-x)^2;x\right)}\leq \dfrac{3(n+2)}{(n+1)^2}.
		\end{align*}
		Therefore, the following inequality holds
		\begin{align*}
			&	\left|K_n(f;x)-f(x)-\dfrac{1-2x}{2(n+1)}f^{\prime}(x)-\dfrac{1}{2}\left[\dfrac{x(1-x)(n-1)}{(n+1)^2}+\dfrac{1}{3(n+1)^2}\right]f^{\prime\prime}(x)\right|\\
			&\leq \left[x(1-x)\dfrac{n-1}{(n+1)^2}+\dfrac{1}{3(n+1)^2}\right]\left\{\dfrac{25}{12h(n+1)}\omega_1(f^{\prime\prime};h)+\left(\dfrac{3}{4}+\dfrac{3(n+2)}{16h^2(n+1)^2}\right)\omega_2(f^{\prime\prime};h)\right\}
		\end{align*}
		and for $h=\dfrac{1}{\sqrt{n+1}}$ we obtain, after multiplying both sides by $n$,
			\begin{align*}
		&\left|n\left[K_n(f;x)-f(x)\right]-\dfrac{n}{n+1}\left(\dfrac{1}{2}-x\right)f^{\prime}(x)-\dfrac{1}{2}\left[x(1-x)\dfrac{n(n-1)}{(n+1)^2}+\dfrac{n}{3(n+1)^2}\right]f^{\prime\prime}(x)\right|	 \\
		&\leq \dfrac{9}{32}\left\{\dfrac{2}{\sqrt{n+1}}\omega_1\left(f^{\prime\prime};\dfrac{1}{\sqrt{n+1}}\right)+\omega_2\left(f^{\prime\prime};\dfrac{1}{\sqrt{n+1}}\right)\right\}.
		\end{align*}
		We can write
		\begin{align*}
		&\left|n\left[K_n(f;x)-f(x)\right]-\dfrac{1-2x}{2}f^{\prime}(x)-\dfrac{1}{2}x(1-x)f^{\prime\prime}(x)\right|\\
		&\leq\left|n\left[K_n(f;x)-f(x)\right]-\dfrac{n}{n+1}\left(\dfrac{1}{2}-x\right)f^{\prime}(x)-\dfrac{1}{2}\left[x(1-x)\dfrac{n(n-1)}{(n+1)^2}+\dfrac{n}{3(n+1)^2}\right]f^{\prime\prime}(x)\right|	 \\
		&+\left|\dfrac{1-2x}{2}\dfrac{1}{n+1}f^{\prime}(x)+\dfrac{1}{2}x(1-x)\dfrac{3n+1}{(n+1)^2}f^{\prime\prime}(x)-\dfrac{n}{6(n+1)^2}f^{\prime\prime}(x)\right|\\
		&\leq\dfrac{9}{32}\left\{\dfrac{2}{\sqrt{n+1}}\omega_1\left(f^{\prime\prime};\dfrac{1}{\sqrt{n+1}}\right)+\omega_2\left(f^{\prime\prime};\dfrac{1}{\sqrt{n+1}}\right)\right\}+\dfrac{2}{3(n+1)}\left(\dfrac{3}{4}\|f^{\prime}\|_{\infty}+\|f^{\prime\prime}\|_{\infty}\right).
		\end{align*}
	\end{proof}

\section{Chebyshev-Gr\"uss inequality for Kantorovich operators}
In a 2011 paper Ra\c sa and the present authors \cite{1} published the following Gr\"uss-type inequality for positive linear operators reproducing constant functions. We give below the improved form of Rusu given in \cite{9}:
\begin{theorem} Let $H:C[a,b]\to C[a,b]$ be positive, linear and satisfy $He_0=e_0$. Put
	$$ D(f,g;x):=H(fg;x)-H(f;x)\cdot H(g;x). $$
	Then for $f,g\in C[a,b]$ and $x\in[a,b]$ fixed one has
	$$ |D(f,g;x)|\leq\dfrac{1}{4}\tilde\omega\left(f;2\sqrt{H\left((e_1-x)^2;x\right)}\right)\cdot \tilde\omega\left(g;2\sqrt{H\left((e_1-x)^2;x\right)}\right). $$
	Here $\tilde\omega$ is the least concave majorant of the first order modulus $\omega_1$ given by
	$$ \tilde{\omega}(f;t)=\sup\left\{\dfrac{(t-x)\omega_1(f;y)+(y-t)\omega_1(f;x)}{y-x}: 0\leq x\leq t\leq y\leq b-a, x\ne y\right\}. $$
\end{theorem}
\begin{remark} For an accesible proof of the equality between $\tilde{\omega}$ and a certain $K$-functional used in the proof of the above theorem see \cite{8}.
	\end{remark}
Hence the non-multiplicativity of Kantorovich operators can be explained as in
\begin{theorem}
	For the classical Kantorovich operators $K_n:C[0,1]\to C[0,1]$ one has the uniform inequality
	\begin{align}
	\|K_n(fg)-K_nfK_ng\|_{\infty}\leq\dfrac{1}{4}\tilde{\omega}\left(f;2\sqrt{\dfrac{1}{2(n+1)}}\right)\tilde{\omega}\left(g;2\sqrt{\dfrac{1}{2(n+1)}}\right),n\geq 1,
	\end{align}
	for all $f,g\in C[0,1]$.
\end{theorem}
\begin{proof}
	The most precise upper bound is obtained if we use the exact representation
	$$ K_n\left((t-x)^2;x\right)=\dfrac{1}{(n+1)^2}\left\{(n-1)x(1-x)+\dfrac{1}{3}\right\}.  $$
	Close to $x=0,1$ this shows the familiar endpoint improvement. For shortness we use the estimate 
	$$ K_n\left((t-x)^2;x\right)\leq\dfrac{1}{2(n+1)}.  $$
\end{proof}

\section{Gr\"uss-Voronovskaya theorems}
The first Gr\"uss-Voronovskaya theorem for classical Bernstein operators was given by Gal and Gonska \cite{4}. In Theorem 2.1 of this paper a quantitative form was given (see also Theorem 2.5 there). The other examples in \cite{4} deal with operators reproducing linear functions; this is not the case for the Kantorovich mappings. The limit for $K_n$ was identified recently in \cite{2} to be the same as in the Bernstein case, namely
$$ f^{\prime}(x)g^{\prime}(x)x(1-x), \textrm{ for } f,g\in C^2[0,1].$$
Our first quantitative version is given in
\begin{theorem}\label{t2.3} Let $f,g\in C^2[0,1]$. Then for each $x\in [0,1]$
	\begin{align*}&\|n\left[K_n(fg)-K_nf\cdot K_ng\right]-Xf^{\prime}g^{\prime}  \|_{\infty}=\left\{\begin{array}{l}
	o(1),\,\, f,g\in C^2[0,1],\\
	\vspace{-0.4cm}\\
	{\cal O}\left(\dfrac{1}{\sqrt{n}}\right),\,\,f,g\in C^3[0,1],\\
	\vspace{-0.4cm}\\
	{\cal O}\left(\dfrac{1}{{n}}\right),\,\,f,g\in C^4[0,1].
	\end{array}\right.
	\end{align*}
	\end{theorem}
 \begin{proof}
 	We proceed as in \cite{4} by creating first three Voronovskaya-type expressions from the difference in question plus the remaining quantities. Recall that the Voronovskaya limit for Kantorovich operators is
 	$$ \dfrac{1}{2}(Xf^{\prime})^{\prime}=\dfrac{1}{2}Xf^{\prime\prime}(x)+\dfrac{1}{2}X^{\prime}f^{\prime}(x), $$
 	where $X:=x(1-x)$, so $X^{\prime}=1-2x$.
 	
 	For $f,g\in C^2[0,1]$ one has
 	\begin{align*}
 	& K_n(fg;x)-K_n(f;x)K_n(g;x)-\dfrac{1}{n}Xf^{\prime}(x)g^{\prime}(x)\\
 	&=K_n(fg;x)-(fg)(x)-\dfrac{1}{2n}\left(X(fg)^{\prime}\right)^{\prime}\\
 	&-f(x)\left[K_n(g;x)-g(x)-\dfrac{1}{2n}(Xg^{\prime})^{\prime}\right]
 	-g(x)\left[K_n(f;x)-f(x)-\dfrac{1}{2n}(Xf^{\prime})^{\prime}\right]\\
 	&+\left[g(x)-K_n(g;x)\right]\left[K_n(f;x)-f(x)\right]\\
 	&-K_n(f;x)K_n(g;x)-\dfrac{1}{n}Xf^{\prime}g^{\prime}+(fg)(x)+\dfrac{1}{2n}(X(fg)^{\prime})^{\prime}\\
 		&+f(x)\left[K_n(g;x)-g(x)-\dfrac{1}{2n}(Xg^{\prime})^{\prime}\right]
 	+g(x)\left[K_n(f;x)-f(x)-\dfrac{1}{2n}(Xf^{\prime})^{\prime}\right]\\
 	&-\left[g(x)-K_n(g;x)\right]\left[K_n(f;x)-f(x)\right].
 	\end{align*}
 	
 	The first three lines will be estimated below. First we will show that the sum of the following three lines equals $0$.
 	
 	For the time being we will leave out the argument $x$. One has 
 	\begin{align*}
 	&-K_nf\cdot K_ng-\dfrac{1}{n}Xf^{\prime}g^{\prime}+fg+\dfrac{1}{2n}\left(X^{\prime}(fg)^{\prime}+X(fg)^{\prime\prime}\right)\\
 	&+fK_ng-fg-\dfrac{1}{2n}f(X^{\prime}g^{\prime}+Xg^{\prime\prime})
 	+gK_nf-fg-\dfrac{1}{2n}g(X^{\prime}f^{\prime}+Xf^{\prime\prime})\\
 	&-[g-K_ng]\cdot [K_nf-f]\\
 	&=-K_nf\cdot K_ng-\dfrac{1}{n}Xf^{\prime}g^{\prime}+fg+\dfrac{1}{2n}\left(X^{\prime}f^{\prime}g+X^{\prime}fg^{\prime}\right)+\dfrac{1}{2n}X\left(f^{\prime\prime}g+2f^{\prime}g^{\prime}+fg^{\prime\prime}\right)\\
 	&+fK_ng-fg-\dfrac{1}{2n}(fX^{\prime}g^{\prime}+fXg^{\prime\prime})
 	+gK_nf-fg-\dfrac{1}{2n}(gX^{\prime}f^{\prime}+gXf^{\prime\prime})\\
 	&-gK_nf+K_ng\cdot K_nf+fg-fK_ng=0.
 	\end{align*}
 For the first two lines above we will use the Voronovskaya estimate given earlier, namely that for $h\in C^2[0,1]$ one has
\begin{align*}
&	\left\|n\left(K_nh-h\right)-\dfrac{1}{2}\left(Xh^{\prime}\right)^{\prime} \right\|_{\infty}\leq \dfrac{2}{3(n+1)}\left(\dfrac{3}{4}\|h^{\prime}\|_{\infty}+\|h^{\prime\prime}\|_{\infty}\right)\nonumber\\
&+\dfrac{9}{32}\left\{\dfrac{2}{\sqrt{n+1}}\omega_1\left(h^{\prime\prime};\dfrac{1}{\sqrt{n+1}}\right)+\omega_2\left(h^{\prime\prime};\dfrac{1}{\sqrt{n+1}}\right)\right\}=:U(h,n).
\end{align*}
For the third line we use Theorem \ref{t1.4} showing that for $h\in C^2[0,1]$ we get
$$ \| K_nh-h \|_{\infty}\leq\dfrac{1}{2n}\| h^{\prime}\|_{\infty}+\dfrac{9}{8n}\| h^{\prime\prime}\|_{\infty}={\cal O}\left(\dfrac{1}{n}\right). $$
Collecting these inequalities gives
\begin{align*}&\|n\left[K_n(fg)-K_nf\cdot K_ng\right]-Xf^{\prime}g^{\prime}  \|_{\infty}
\leq U(fg,n)+\|f\|_{\infty}U(g,n)+\| g\|_{\infty}U(f,n)+{\cal O}\left(\dfrac{1}{n}\right)\\
&=\left\{\begin{array}{l}
o(1),\,\, f,g\in C^2[0,1],\\
\vspace{-0.4cm}\\
{\cal O}\left(\dfrac{1}{\sqrt{n}}\right),\,\,f,g\in C^3[0,1],\\
\vspace{-0.4cm}\\
{\cal O}\left(\dfrac{1}{{n}}\right),\,\,f,g\in C^4[0,1].
\end{array}\right.
\end{align*}
\end{proof}

In the following we  give a Gr\"uss-Voronovskaya type theorem when $f$ and $g$ are only in $C^1[0,1]$.

\begin{theorem} Let $f,g\in C^1[0,1]$ and $n\geq 1$. Then there is a constant $C$ independent of $n,f,g$ and $x$, such that
	\begin{align*}
	&\left\|K_n(fg)-K_nf\cdot K_ng-\dfrac{X}{n}f^{\prime}g^{\prime}\right\|_{\infty}\leq \dfrac{C}{n}\left\{\omega_3\left(f^{\prime},n^{-\frac{1}{6}}\right)\omega_3\left(g^{\prime},n^{-\frac{1}{6}}\right)\right.\\
	&+\|f^{\prime}\|_{\infty}\omega_3\left(g^{\prime},n^{-\frac{1}{6}}\right)+\|g^{\prime}\|_{\infty}\omega_3\left(f^{\prime},n^{-\frac{1}{6}}\right)\\
	&+\left.
	\max\left\{ \dfrac{\|f^{\prime}\|_{\infty}}{n^{\frac{1}{2}}},\omega_3\left(f^{\prime},n^{-\frac{1}{6}}\right)\right\}
	\max\left\{ \dfrac{\|g^{\prime}\|_{\infty}}{n^{\frac{1}{2}}},\omega_3\left(g^{\prime},n^{-\frac{1}{6}}\right)\right\}\right\}.
	\end{align*}
	\end{theorem}
\begin{proof}
	Let
	\begin{equation}
	\label{x1} E_n(f,g;x)=K_n(fg;x)-K_n(f;x)K_n(g;x)-\dfrac{x(1-x)}{n}f^{\prime}(x)g^{\prime}(x),
		\end{equation}
and denote $C$ a constant independent of $n,f,g$ and $x$, which may change its values during the course of the proof.

For $f,g\in C^1[0,1]$ fixed and $u,v\in C^4[0,1]$ arbitrary, one has
\begin{align}\label{x4}
|E_n(f,g;x)|&=\left|E_n(f-u+u,g-v+v;x)\right|\\
&\leq\left|E_n(f-u,g-v;x)\right|+\left|E_n(u,g-v;x)\right|+
\left|E_n(f-u,v;x)\right|+\left|E_n(u,v;x)\right|.\nonumber
\end{align}
Let $h(x)=x,\,x\in[0,1]$. Applying \cite[Theorem 4]{1} there exists $\eta,\theta\in[0,1]$ such that
\begin{align}\label{x2}
K_n(fg;x)-K_n(f;x)K_n(g;x)&=f^{\prime}(\eta)g^{\prime}(\theta)
\left[K_n(h^2;x)-(K_n(h;x))^2\right]\nonumber\\
&=f^{\prime}(\eta)g^{\prime}(\theta)\left\{x(1-x)\dfrac{n}{(n+1)^2}+\dfrac{1}{12(n+1)^2}\right\}.\end{align}
From (\ref{x1}) and (\ref{x2}) we get
\begin{align}\label{x5}
\left|nE_n(f,g;x)\right|&\leq \left[x(1-x)\dfrac{n^2}{(n+1)^2}+\dfrac{n}{12(n+1)^2}+x(1-x)\right]\|f^{\prime}\|_{\infty}\|g^{\prime}\|_{\infty}\nonumber\\
&\leq2\left[x(1-x)+\dfrac{1}{24(n+1)}\right]\|f^{\prime}\|_{\infty}\|g^{\prime}\|_{\infty}.
\end{align}
Using Theorem \ref{t1.6} for $f\in C^4[0,1]$ we get
\begin{align*}
\left|n\left[K_n(f;x)-f(x)\right]-\dfrac{1}{2}\left(Xf^{\prime}\right)^{\prime}(x)\right|\leq C\dfrac{1}{n}\left(\|f^{\prime}\|_{\infty}+\|f^{\prime\prime}\|_{\infty}+\|f^{\prime\prime\prime}\|_{\infty}+\|f^{(4)}\|_{\infty}\right).
\end{align*}
But, for $f\in C^n[a,b]$, $n\in\mathbb{N}$ one has (see \cite[Remark 2.15]{G_teza})
$$ \displaystyle\max_{0\leq k\leq n}\left\{\| f^{(k)}\|\right\}\leq C\max\left\{
\|f\|_{\infty},\|f^{(n)}\|_{\infty}\right\}. $$
Therefore,
\begin{equation}\label{x3}
\left|n\left[K_n(f;x)-f(x)\right]-\dfrac{1}{2}\left(Xf^{\prime}\right)^{\prime}(x)\right|\leq \dfrac{C}{n}\max\left\{
\|f^{\prime}\|_{\infty},\|f^{(4)}\|_{\infty}\right\}
\end{equation}
For $u,v\in C^4[0,1]$ using the same decomposition as in proof of Theorem \ref{t2.3},  the relation (\ref{x3}) and Theorem \ref{t1.4}, we get
\begin{align}\label{x6}
\left|E_n(u,v;x)\right|	&\leq \left|K_n(uv;x)-(uv)(x)-\dfrac{1}{2n}\left(X(uv)^{\prime}\right)^{\prime}\right|\nonumber\\
&+|u(x)|\left|K_n(v;x)-v(x)-\dfrac{1}{2n}(Xv^{\prime})^{\prime}\right|\
+|v(x)|\left|K_n(u;x)-u(x)-\dfrac{1}{2n}(Xu^{\prime})^{\prime}\right|\nonumber\\
&+\left|v(x)-K_n(v;x)\right|\left|K_n(u;x)-u(x)\right|\nonumber\\
&\leq\dfrac{C}{n^2}\max\left\{
\|u^{\prime}\|_{\infty},\|u^{(4)}\|_{\infty}\right\}\max\left\{
\|v^{\prime}\|_{\infty},\|v^{(4)}\|_{\infty}\right\}.
\end{align}
From the relations (\ref{x4}), (\ref{x5}) and (\ref{x6}) we obtain
\begin{align*}
|E_n(f,g;x)|&\leq \dfrac{2}{n}\left[x(1-x)+\dfrac{1}{24(n+1)}\right]\left\{
\|(f-u)^{\prime}\|_{\infty}\|(g-v)^{\prime}\|_{\infty}+\|u^{\prime}\|_{\infty}\|(g-v)^{\prime}\|_{\infty}\right.\\
&\left.+\|(f-u)^{\prime}\|_{\infty}\|v^{\prime}\|_{\infty}\right\}+\dfrac{C}{n^2}\max\left\{
\|u^{\prime}\|_{\infty},\|u^{(4)}\|_{\infty}\right\}\max\left\{
\|v^{\prime}\|_{\infty},\|v^{(4)}\|_{\infty}\right\}.
\end{align*}
Using \cite[Lemma 3.1]{Rocky} for $r=1$, $s=2$, $f_{h,3}=u$ and $g_{h,3}=v$, for all $h\in (0,1]$ and $n\in \mathbb{N}$, it follows
\begin{align*}|E_n(f,g;x)|&\leq \dfrac{C}{n}\left\{\omega_3(f^{\prime},h)\omega_3(g^{\prime},h)+\dfrac{1}{h}\omega_1(f,h)\omega_3(g^{\prime},h)+\dfrac{1}{h}\omega_1(g,h)\omega_3(f^{\prime},h)\right\}\\
&+\dfrac{C}{n^2}\max\left\{\dfrac{1}{h}\omega_1(f,h),\dfrac{1}{h^3}\omega_3(f^{\prime},h)\right\}\cdot \max\left\{\dfrac{1}{h}\omega_1(g,h),\dfrac{1}{h^3}\omega_3(g^{\prime},h)\right\}\\
&\leq \dfrac{C}{n}\left\{\omega_3(f^{\prime},h)\omega_3(g^{\prime},h)+
\|f^{\prime}\|_{\infty}\omega_3(g^{\prime},h)+\|g^{\prime}\|_{\infty}\omega_3(f^{\prime},h)\right\}\\
&+\dfrac{C}{n^2}\max\left\{\|f^{\prime}\|_{\infty},\dfrac{1}{h^3}\omega_3(f^{\prime},h)\right\}\cdot \max\left\{\|g^{\prime}\|_{\infty},\dfrac{1}{h^3}\omega_3(g^{\prime},h)\right\}.
\end{align*}
Choosing $h=n^{-\frac{1}{6}}$, we obtain
\begin{align*}|E_n(f,g;x)|&\leq \dfrac{C}{n}\left\{\omega_3\left(f^{\prime},n^{-\frac{1}{6}}\right)\omega_3\left(g^{\prime},n^{-\frac{1}{6}}\right)\right.\\
&+\|f^{\prime}\|_{\infty}\omega_3\left(g^{\prime},n^{-\frac{1}{6}}\right)+\|g^{\prime}\|_{\infty}\omega_3\left(f^{\prime},n^{-\frac{1}{6}}\right)\\
&+\left.
\max\left\{ \dfrac{\|f^{\prime}\|_{\infty}}{n^{\frac{1}{2}}},\omega_3\left(f^{\prime},n^{-\frac{1}{6}}\right)\right\}
\max\left\{ \dfrac{\|g^{\prime}\|_{\infty}}{n^{\frac{1}{2}}},\omega_3\left(g^{\prime},n^{-\frac{1}{6}}\right)\right\}\right\}.
\end{align*}
This implies the theorem.

\end{proof}

\noindent{\bf Acknowledgements.} The first author was supported by  Lucian Blaga University of Sibiu research under grant LBUS-IRG-2018-04. The second appreciates financial support of the University of Duisburg-Essen.

$  $

\noindent{\bf Author details}\\

\noindent${}^{1}$Lucian Blaga University of Sibiu, Department of Mathematics and Informatics, Str. Dr. I. Ratiu, No.5-7, RO-550012  Sibiu, Romania, e-mail: anamaria.acu@ulbsibiu.ro\\
${}^{2}$University of Duisburg-Essen, Faculty of Mathematics, Bismarckstr. 90, 47057 Duisburg, Germany, e-mail: heiner.gonska@uni-due.de;



\begin{thebibliography}{99}
	
	\bibitem{1} A.M. Acu,  H. Gonska, I. Ra\c sa, Gr\"uss-type and Ostrowski-type inequalities in approximation theory, Ukrainian Math. J., {\bf 63}, No. 6, 843-864 (2011).
	
	\bibitem{2} A.M. Acu, N. Manav, D.F. Sofonea, 
	Approximation properties of $\lambda$-Kantorovich operators, 
	J. Inequal. Appl. 2018, 2018:202.
	
	\bibitem{3} A.M. Acu, I. Ra\c sa,  New estimates for the differences of positive linear operators, Numer. Algorithms, {\bf 73}, No. 3,  775-789 (2016).
	
	\bibitem{Dzy} V.K. Dzyadyk,  Introduction to the theory of uniform approximation of functions by polynomials (Russian), Izdat. "Nauka'', Moscow, (1977), 511 pp.
	
	\bibitem{4} S. Gal, H. Gonska,
	Gr\" uss and Gr\" uss-Voronovskaya-type estimates for some Bernstein-type polynomials of real and complex variables,
	Jaen J. Approx., {\bf 7}, No. 1, 97-122 (2015).
	
	
	\bibitem{G_teza} H. Gonska,  Quantitative Aussagen zur Approximation durch positive lineare
	Operatoren, Ph.D. Dissertation, University of Duisburg, (1979).
	
		\bibitem{Rocky}  H. Gonska, Degree of approximation by lacunary interpolators : (0, ...,R-2,R) interpolation, Rocky Mount. J. Math., {\bf 19}, No. 1, 157-171, (1989).
	
	\bibitem{5} H. Gonska, M. Heilmann, I. Ra\c sa, 
	Kantorovich operators of order k, 
	Numer. Funct. Anal. Optim., {\bf 32}, No. 7, 717-738 (2011).
	
\bibitem{6} H.	Gonska, I. Ra\c sa, A Voronovskaya estimate with second order modulus
	of smoothness. In: D. Acu et al. (eds.): Mathematical Inequalities
	(Proc. 5th Int. Sympos., Sibiu 2008), Sibiu: Publishing House of "Lucian
	Blaga" University, 76-90 (2009).
	
\bibitem{7} H. Gonska, X.-l. Zhou,
The strong converse inequality for Bernstein-Kantorovich operators, 
Concrete analysis, Comput. Math. Appl., {\bf 30}, No. 3-6, 103-128  (1995).

	\bibitem{K} L.V. Kantorovich, {\it Sur certains developpements suivant les polyn\^{o}mes de la
	forme de S. Bernstein I, II}, Dokl. Akad. Nauk. SSSR , 563-568, 595-600 (1930).

\bibitem{Pee} J. Peetre, On the connection between the theory of interpolation spaces and approximation theory, Proceedings of the Conference on the Constructive Theory of Functions (Approximation Theory) (Budapest, 1969), pp. 351-363, Akad\'emiai Kiad\'o, Budapest, 1972.


\bibitem{8} R. P\u alt\u anea,  Representation of the $K$-functional
$K(f,C[a,b],C^1[a,b],⋅)$ - a new approach, Bull. Transilv. Univ. Bra\c sov Ser. III 3(52), 93-99 (2010).

\bibitem{9} M.D. Rusu, { Chebyshev-Gr\"uss- and Ostrowski-type Inequalities}, PhD Thesis, Duisburg-Essen University, 2014.

\bibitem{Vid} V.S., Videnskij, Bernstein polynomials. Textbook. (Mnogochleny Bernshtejna. Uchebnoe posobie.) (Russian), 
Leningrad: Leningradskij Pedagogicheskij Institut Im. A. I. Gertsena, 64 p. (1990).


\bibitem{Zyg} A. Zygmund, Smooth functions, Duke Math. J., {\bf 12}, 47-76 (1945).



	
	
	
	
	
	
	
	

	
\end{thebibliography}
\end{document}